\documentclass[10pt,a4paper]{article}
\usepackage{latexsym}
\usepackage{amsmath,amssymb,amsfonts}
\renewcommand{\baselinestretch}{1}
\usepackage[top=2cm,bottom=2cm,left=2cm,right=2cm]{geometry}

\newtheorem{prethm}{{\bf Theorem}}

\newenvironment{thm}{\begin{prethm}{\hspace{-0.5
               em}{\bf.}}}{\end{prethm}}

\newtheorem{prepro}[prethm]{Proposition}

\newtheorem{prelem}[prethm]{Lemma}
\newenvironment{lem}{\begin{prelem}{\hspace{-0.5
               em}{\bf.}}}{\end{prelem}}

\newtheorem{precor}[prethm]{Corollary}

\newtheorem{preque}[prethm]{Question}
\newenvironment{que}{\begin{preque}{\hspace{-0.5
               em}{\bf.}}}{\end{preque}}
\newtheorem{precon}[prethm]{Conjecture}

\newtheorem{preremark}{{\bf Remark}}

\newtheorem{preexample}{{\bf Example}}

\newenvironment{example}{\begin{preexample}\em{\hspace{-0.5
               em}{\bf.}}}{\end{preexample}}

\newtheorem{preproblem}{{\bf problem}}

\newtheorem{preproof}{{\bf Proof.}}

\newenvironment{proof}[1]{\begin{preproof}{\rm
               #1}\hfill{$\Box$}}{\end{preproof}}

\renewcommand{\thefootnote}

\begin{document}
\title{Certain Simple Maximal Subfields in Division Rings}
\author{M. Aaghabali$^{\,\rm a}$, M.H. Bien$^{\,\rm b}$\\
{\footnotesize{\em $^{\rm a}$ School of Mathematics, University of Edinburgh,
Edinburgh EH9 3JZ, Scotland}}\\
{\footnotesize {\em $^{\rm b}$Faculty of Mathematics and Computer Science,
University of Science, VNU-HCM,
227 Nguyen Van Cu Str.,}}\\
{\footnotesize{\em  Dist. 5, HCM-City, Vietnam}}}
\footnotetext{E-mail Addresses: {\tt maghabali@gmail.com}, {\tt mhbien@hcmus.edu.vn}\\ The first author's current institute: School of Mathematics, Statistics and Computer Science, University of Tehran, Tehran, Iran}
\date{}
\maketitle
\begin{quote}
{\small \hfill{\rule{13.3cm}{.1mm}\hskip2cm} \textbf{Abstract}
Let $D$ be a division ring finite dimensional over its center $F$. The goal of this paper is to prove that for any positive integer $n$ there exists $a\in D^{(n)},$ the $n$-th multiplicative derived subgroup, such that $F(a)$ is a maximal subfield of $D$. We also show that a single depth-$n$ iterated additive commutator would generate a maximal subfield of $D.$
\vspace{1mm} {\renewcommand{\baselinestretch}{1}
\parskip = 0 mm

}

}

\vspace{-3mm}\hfill{\rule{13.3cm}{.1mm}\hskip2cm}
\end{quote}
\section{Preliminary}
  Throughout this paper $D$ is a division ring with center $F.$  An element $a\in D$ is called \textit{algebraic}   over $F$, if there exists a non-zero polynomial $a_0+a_1x+\dots+a_nx^n$ over $F$ such that $a_0+a_1a+\dots+a_na^n=0.$ If $a\in D$, then $F(a)$ denotes the subfield of $D$  generated by $F$ and $\{a\}$. For a (multiplicative) group $G$ we denote by $(a,b)=aba^{-1}b^{-1}$ the multiplicative commutator of $a,b\in G$ and $(G,G)$ the multiplicative commutator subgroup of $G$. We denote by $G\supseteq G'\supseteq \dots \supseteq G^{(n)}\supseteq\dots$ the derived series of $G$, that is, $G'=(G,G)$, and $G^{(n+1)}=(G^{(n)},G^{(n)})$ for every $n\ge 1$. For a unital associative ring $R$ we use $[a,b]=ab-ba$ to denote the additive commutator of $a,b\in R$ and $R_1=[R,R]$ the additive commutator subgroup of $R$. We denote by $R\supseteq R_1\supseteq \dots \supseteq R_n\supseteq\dots$ the additive derived series of $R$, that is, $R_1=[R, R]$, and $R_{n+1}=[R_n,R_n]$ for every $n\ge 1$. For a given division ring $D$ we call $D^{(n)}$ and $D_n$ the $n$-th multiplicative and additive derived groups of $D,$ respectively. In the case of division ring $D$ we simply use $D'$ and $[D,D]$ to denote the multiplicative and additive group of commutators in $D,$ respectively.  If $A$ is a subset of $D$ we use $A^*$ to denote $A\setminus\{0\}.$ A subfield $K$ of $D$ is called a {\it maximal subfield} if $K$ is its own centralizer in $D^*.$ We denote by ${\rm dim}_FD$ the dimension of $D$ over $F.$ If ${\rm dim}_FD=n^2,$ then $n$ is called the {\it degree} of division ring $D.$ By $M_n(K), GL_n(K)$ and $SL_n(K)$ we mean all square matrices, all invertible matrices and all matrices of determinant one of order $n$ with entries from $K,$ respectively.

 Mahdavi-Hezavehi in~\cite{mahdavi1} investigates the algebraic properties of the multiplicative group of commutators in a division ring and shows that any subfield $K$ of a division ring $D$ which is separable over the center of $D$ is generated over the center by a commutator subgroup of $D'$. Afterwards, Mahdavi-Hezavehi and his colleagues in~\cite{mahdavi2} studied other generating properties of commutator subgroup and showed that each finite separable extension of the center of $D$ could be considered as a simple extension $F(c),$ where $c$ is an element in $D'.$ Now, it is natural to consider similar questions in terms of some other elements coming from certain substructures of a division ring. In particular, one can pose the following questions:
 \begin{que} {\rm\cite[Problems 28,~29]{m}} {\rm Let $D$ be a division ring finite dimensional over its center $F.$

 (i) Whether for any non-central normal subgroup $N$ of $D^*$ does there exist element $c\in N$ such that $F(c)$ is a maximal subfield in $D?$

 (ii) Whether for any non-central subnormal subgroup $N$ of $D^*$ does there exist element $c\in N$ such that $F(c)$ is a maximal subfield in $D?$}
 \end{que}
 In this note we rely on rational identities to show that some maximal subfields are generated by elements coming from $D^{(n)}$ and $D_n,$ resp. $n$-th derived subgroup of $D^*$ and $n$-th iterated group of additive commutators, for any positive integer $n.$ These fall under a wider class of problems concerning the question of whether a non-central subnormal subgroup of $D^*$ cannot be ``too small", and questions about the images of (non-commutative) polynomials evaluated on central simple algebras. In the case $n= 1,$ both results have been proved Chebotar et al in~\cite[Theorem 3, theorem 6]{chebo}, and recently again by the authors and S. Akbari in~\cite[Theorem 6, Theorem 7]{aa}.
Both \cite{aa, chebo} and the current paper use rational polynomial identities in proving the aforementioned results.  The idea is simple and clever:  The key is a certain (non-commutative) polynomial $g_n(x, y_1,\dots , y_n)$ that vanishes whenever an algebraic element of degree $\leq n$ is substituted into $x.$  One takes $n<deg D,$ substitutes a relevant rational expression into $x$ and proves that the resulting expression cannot vanish on $D\otimes_FL,$ where $L$ is some splitting field of $D.$  In~\cite{aa} and~\cite{chebo}, the expressions substituted into $x$ are single additive, resp. multiplicative, commutators on two variables, whereas here, iterated commutators are considered.

\section{Rational identities}
	Let $F$ be a field and $X=\{x_1,\dots,x_m\}$ be $m$ noncommuting indeterminates. Denote by $F\langle X\rangle$ and $F(X)$ respectively the free algebra in $X$ over $F$ and the universal division ring of fractions of $F\langle X\rangle$. A {\it rational expression} over $F$ is an element of $F(X)$. Let $R$ be an $F$-algebra. A rational expression $f$ over $F$ is said to be a {\it rational identity} of $R$ if it vanishes on all permissible substitutions from $R$. In this case, we say that $R$ {\it satisfies the rational identity} $f=0$.
\vspace{-.2cm}	
	\begin{example}{\rm  \begin{enumerate}
				\item It is not hard to see that (Hua's identity) $(x^{-1} +(y^{-1} -x^{-1})^{-1})^{-1}-x +xyx=0$ is a rational identity of every algebra over an arbitrary field $F$.
\vspace{-.2cm}
				\item One can easily verify that $(x+y)^{-1}-y^{-1}(x^{-1}+y^{-1})^{-1}x^{-1}=0$ is a rational identity of every algebra over an arbitrary field $F$.
\vspace{-.2cm}
				\item It is easy to check that $((x, (y,z)x(y,x)^{-1})^3, z)=0$ vanishes on permissible substitutions of $M_3(F)$ for any field $F$.
			\end{enumerate}
		}\end{example}
\vspace{-.2cm}
		
	A rational identity $f$ of an algebra $R$ is called {\it nontrivial}\index{nontrivial} if $f$ is non-zero in $F(X)$ \cite{chebo}. In the special case when $R=D$ is a division ring, we have some further information: assume that $f=0$ is a rational identity of $D$. Then $f$ is non-trivial if and only if there exists a division ring $L$ containing all coefficients of $f$ and $f$ is not a rational identity of $L$. One direction of the statement is trivial, to see the other direction assume that $f$ is nontrivial. Then, it is well known that there exists a division ring $L$ with infinite center which contains $F$, and $L$ is infinite dimensional over its center. Hence by \cite{chiba}, $f=0$ is not a rational identity of $D$.  In the example, it is easily seen that (1) and (2) are trivial, however one can verify that (3) is nontrivial.
	
	In this paper, our algebras $R$ are central simple algebras over a field $F$. That is,  $R\cong M_n(D)$ where $D$ is a division ring which is finite dimensional over $F$. We denote by $\mathcal{I}(R)$ the set of all nontrivial rational identities of the algebra $R$. It is known that a division ring $D$ with infinite center $F$ is a finite dimensional vector space over its center  if and only if $\mathcal{I}(D)\ne \emptyset$ \cite{chiba}. Therefore, there are rings $R$ with $\mathcal{I}(R)=\emptyset.$ Moreover,
	
	\begin{thm}\label{t2.2}{\rm \cite[Theorem 11]{Pa_Am_66}} Let $F$ be an infinite field and $R$ be a central simple $F$-algebra with ${\rm dim}_FR=n^2$. Assume that $L$ is a field extension of $F$. Then $\mathcal{I}(R)=\mathcal{I}(M_n(F))=\mathcal{I}(M_n(L))$.
	\end{thm}
	
	We consider the following example of a rational expression which is important in this paper. Given a positive integer $n$ and $n+1$ noncommutative indeterminates $x,y_1,\dots, y_n$, put $$g_n(x,y_1,\dots, y_n)=\sum\limits_{\delta  \in {S_{n + 1}}} {{\rm sign}(\delta ){x^{\delta (0)}}{y_1}{x^{\delta (1)}}{y_2}{x^{\delta (2)}} \dots {y_n}{x^{\delta (n)}}}, $$ where $S_{n+1}$ is the symmetric group of $\{\,0,\dots, n\,\}$ and ${\rm sign}(\delta)$ is the sign of permutation $\delta$. This is a rational expression defined in \cite{Bo_BeMaMi_96} as a mean to test whether an element is algebraic of degree $n.$ This rational expression may be considered as a generalisation of the characteristic polynomials of matrices of degree $n$ over a field.
	
\begin{lem}\label{2.2} Let $R=M_n(D)$ be a central simple algebra over its center $F$. For any element $a\in R$, the following conditions are equivalent.
\vspace{-.2cm}
		\begin{enumerate}
			\item The element $a$ is algebraic over $F$ of degree less than or equal to $n$.
\vspace{-.2cm}
			\item $g_n(a,r_1,r_2,\dots, r_n)=0$ for any $r_1, r_2,\dots, r_n\in R$.
		\end{enumerate}
	\end{lem}
	\begin{proof}{ It is just a corollary of \cite[Corollary 2.3.8]{Bo_BeMaMi_96}.}
	\end{proof}
	
\section{Subfields generated by the elements in the $n$-th multiplicative derived subgroup}	

Let $n$ be a positive integer and let $x_1,\dots,x_{2^n}$ be $2^n$ indeterminates. We will define a special rational polynomial  $u_n(x_1,\dots,x_{2^n})$ successively as follows: set $u_1(x_1,x_2)=(x_1,x_2)=x_1x_2x_1^{-1}x_2^{-1}$ and assume that  $u_{n-1}(x_1,\dots,x_{2^{n-1}})$ is defined. Then we put $$u_n(x_1,\dots,x_{2^n})=u_1(u_{n-1}(x_1,\dots,x_{2^{n-1}}),u_{n-1}(x_{2^{n-1}+1},\dots,x_{2^n})).$$ This polynomial relates to the solvability of a group: if $G$ is a solvable group of length $\le n$, that is $G^{(n)}=1$, then $u_n(a_1,\dots,a_{2^n})=1$ for every $a_1,\dots,a_{2^n}\in G$. In fact, we show the following result.
\begin{lem}\label{l3.1} Let $u_n$ be as above.  If $G$ is a group with (multiplicative) derived series $$G\supseteq G'\supseteq \dots\supseteq G^{(n)}\supseteq\dots,$$ then $u_n(a_1,\dots, a_{2^n})\in G^{(n)}$ for $a_1,\dots, a_{2^n}\in G$.
\end{lem}
\begin{proof}{ We prove the lemma by induction on $n$. Assume that $G$ is a group and $a_1,a_2\in G$. One has $a_1a_2a_1^{-1}a_2^{-1}\in G'$, which implies that $u_1(a_1,a_2)\in G'$. Hence, the lemma holds for $u_1$ and for the group $G$. Assume that for every group $H$, $u_{n-1}(a_1,\dots,a_{2^{n-1}})\in H^{(n-1)}$ for every $a_1,\dots,a_{2^{n-1}}\in H$. We must prove that for every group $G~,u_{n}(b_1,\dots,b_{2^{n}})\in G^{(n)}$ for every $b_1,\dots,b_{2^n}\in G$. This follows immediately from the definitions of $u_n$ and $G^(n)$ by induction on $n.$}
\end{proof}

\begin{lem} \label{l3.2} Let $K$ be an infinite field and $m>1$. For any positive integer $n$ and every non-scalar matrix $C\in SL_m(K)$, there exist non-scalar matrices $A_1, \dots , A_{2^{n}}\in SL_m(K)$ such that $C=u_n(A_1, \dots , A_{2^{n}}).$
\end{lem}
\begin{proof}{We show the lemma by induction on $n$. Assume that $n=1$. It is well-known that every non-scalar matrix in $SL_m(K)$ is a single commutator \cite{thom}. Hence there exist non-scalar matrices $A_1,A_2\in SL_m(K)$ such that $C=A_1A_2A_1^{-1}A_2^{-1}$. Thus, the statement holds in case $n=1$.
	 Assume that the statement is true for $n-1$, that is, for every non-scalar matrix $C\in SL_m(K)$ there exist non-scalar matrices $A_1, \dots , A_{2^{n-1}}\in SL_m(K)$ such that $C=u_{n-1}(A_1, \dots , A_{2^{n-1}}).$ Now by the induction hypothesis for every non-scalar matrix $C,$ there exist non-scalar matrices $B_1,B_2,A_1,\dots,A_{2^n}\in SL_m(K)$ such that  $C=u_1(B_1,B_2),~B_1=u_{n-1}(A_1,\dots,A_{2^{n-1}}),$ and $B_2=u_{n-1}(A_{2^{n-1}+1},\dots,A_{2^n}).$ Therefore, $$\begin{array}{ccc}C=u_1(B_1,B_2)\\~~~~~~~~~~~~~~~~~~~~~~~~~~~~~~~~~~~~~~~~~~~~~~~~~~~~~~=u_1(u_{n-1}(A_1,\dots,A_{2^{n-1}}),u_{n-1}(A_{2^{n-1}+1},\dots,A_{2^n}))\\
~~~~~~~~~~~=u_n(A_1,\dots,A_{2^{n}}).\end{array}$$ This implies that the statement is true in case $n$.} 	
\end{proof}

Before showing the main result of this section, we recall the following well-known lemma.
\begin{lem}\label{l3.3}{\rm \cite[Page 242]{lam}}
	Let $D$ be a division ring with center $F$ and $K$ be a subfield of $D$ containing $F$. If ${\rm dim}_FD=m^2,$ then ${\rm dim}_FK\le m$. The equality holds if and only if $K$ is a maximal subfield of $D$.
\end{lem}

\begin{thm}\label{t3.4} Let $D$ be a  division ring finite dimensional over a its center $F.$ For any positive integer $n$ there exists $a\in D^{(n)},$ the $n$-th multiplicative derived subgroup, such that $F(a)$ is a maximal subfield of $D$.
\end{thm}

\begin{proof} {If $F$ is finite, then $D$ is also finite and we have nothing to prove. Suppose that $F$ is infinite and ${\rm dim}_F D=m^2$. By Lemma~\ref{l3.3}, it suffices to show that there exists $a\in D^{(n)}$ such that ${\rm dim}_FF(a)\ge m$. Indeed, put $$\ell={\rm max}\{\,{\rm dim}_FF(u_n(a_1,\dots,a_{2^n}))\mid a_1,\dots,a_{2^n}\in D^*\,\}.$$ Applying Lemma~\ref{2.2} we see that $g_\ell(u_n(a_1,\dots,a_{2^n}), r_1,\dots ,r_\ell )=0,$ for any $r_1,\dots , r_\ell\in D$ and $a_1,\dots,a_{2^n}\in D^*$. In other words, $$g_\ell(u_n(x_1,\dots,x_{2^n}), y_1,\dots, y_\ell)=0$$ is a rational identity of $D.$ It is not hard to verify that $g_\ell(u_n(x_1,\dots,x_{2^n}), y_1,\dots, y_\ell)$ is a non-zero element of $F(x_1,\dots,x_{2^n},y_1,\dots,y_\ell),$ (see~\cite[Theorem 3.4]{hai}). Hence, by Theorem~\ref{t2.2} it is also a rational identity of $M_n(F)$. This yields that $g_\ell(u_n(A_1,\dots,A_{2^n}),B_1,\dots,B_\ell)=0,$ for all matrices $A_i\in GL_m(F)$ and $B_i\in M_m(F).$ In the view of Lemma~\ref{2.2}, $u_n(A_1,\dots,A_{2^n})$ is algebraic over $F$ of degree $\le \ell$ for every $A_1,\dots,A_{2^n}\in M_m(F)$.		 Now consider the $m\times m$-matrix $T=(t_{ij})_{1\le i,j\le m}$ as follows: if $j=i$ or $j=i+1$, then $t_{ij}=1$; otherwise $t_{ij}=0$. It is easy to see that $T\in SL_m(F)$ and $T$ is algebraic of degree $m$ over $F$. By Lemma~\ref{l3.2}, one can find matrices $A_1,\dots,A_{2^n}\in SL_m(F)$ such that $u_n(A_1,\dots,A_{2^n})=T$. Hence, $\ell\ge m$. This completes the proof.}
\end{proof}
\section{Subfields generated by the elements in the $n$-th additive derived subgroup}	

Let $n$ be a positive integer and let $x_1,\dots,x_{2^n}$ be $2^n$ indeterminates. We define a polynomial  $v_n(x_1,\dots,x_{2^n})$ successively as follows: set $v_1(x_1,x_2)=[x_1,x_2]=x_1x_2-x_2x_1$. Assume that  $v_{n-1}(x_1,\dots,x_{2^{n-1}})$ is defined. Then we put $$v_n(x_1,\dots,x_{2^n})=v_1(v_{n-1}(x_1,\dots,x_{2^{n-1}}),v_{n-1}(x_{2^{n-1}+1},\dots,x_{2^n})).$$
\begin{lem}\label{l4.1} Let $R$ be an algebra with additive derived series $$R\supseteq R_1\supseteq \dots \supseteq R_n\supseteq \dots.$$ If $v_n$ is defined as above, then $v_n(a_1,\dots, a_{2^n})\in R_n$ for $a_1,\dots, a_{2^n}\in R$.
\end{lem}
\begin{proof}{The proof is similar to that of Lemma~\ref{l3.1}.}
\end{proof}

\begin{lem} \label{l4.2} Let $K$ be a field and $m>1$ such that ${\rm char}K\nmid m.$ For any positive integer $n$, and every matrix $C\in M_m(K)$ with zero-trace, there exist matrices $A_1, \dots , A_{2^{n}}\in M_m(K)$ whose trace is zero and $C=v_n(A_1,\dots , A_{2^{n}}).$
\end{lem}
\begin{proof} {The idea of the proof is similar to that of Lemma~\ref{l3.2}. 	
	We prove the lemma by induction on $n$. Assume that $n=1$. In the view of \cite{Pa_AlMu_57}, which states that every matrix in $M_n(K)$ with zero-trace is a single additive commutator, there exist $A_1,A_2\in M_m(K)$ such that $C=A_1A_2-A_2A_1$. Set $B_1=A_1-\frac{tr(A_1)}{m}I_m$ and $B_2=A_2-\frac{tr(A_2)}{m}I_m,$ where by $tr(A)$ we mean the trace of $A.$ We have $C=B_1B_2-B_2B_1=v_1(B_1,B_2)$ and $tr(B_1)=tr(B_2)=0$. Hence, the statement holds in case $n=1$. The general case follows by induction on $n,$ similarly to the proof of Lemma~\ref{l3.2}.} 	
\end{proof}

The following result is the goal of this section. Note that the same result as Theorem~\ref{t3.4},  however,  follows from the case $n=1,$  proved  in~\cite{aa}  and~\cite{chebo},  since  by  a  theorem  of  Amitsur  and  Rowen $D_1=D_2=D_3=\dots~$\cite{amitsur}. However, the proof yields the slightly stronger claim that a single depth-$n$ iterated additive commutator would generate a maximal subfield, which does not follow from $D_1=D_2=D_3=\dots.$
\begin{thm}\label{t4.3} Let $D$ be a division ring finite dimensional over its center $F$ of characteristic either zero or a prime $p$ such that $p\nmid {\rm dim}_FD.$ For any positive integer $n$, there exists a depth-$n$ iterated additive commutator which generates a maximal subfield of $D.$
\end{thm}

\begin{proof} {First note that if ${\rm char}F=0,$ then by a result due to Amitsur and Rowen~\cite{amitsur} we have $D_1=D_2=\cdots.$ Hence, in this case the result follows from~\cite[Theorem 7]{aa}. In the case of ${\rm char}F=p>0,$ the proof is similar to the one of Theorem~\ref{t3.4}. We assume that $F$ is infinite since if $F$ is finite, then $D$ is also finite and there is nothing to prove. Suppose ${\rm dim}_F D=m^2$. In the view of Lemma~\ref{l3.3}, it suffices to show that there exists $a\in D_n$ such that ${\rm dim}_FF(a)\ge m$. Indeed, put $$\ell={\rm max}\{\,{\rm dim}_FF(v_n(a_1,\dots,a_{2^n}))\mid a_1,\dots,a_{2^n}\in D^*\,\}.$$ According to Lemma~\ref{2.2} we see that $g_\ell(v_n(a_1,\dots,a_{2^n}), r_1,\dots ,r_\ell )=0,$ for any $r_1,\dots , r_\ell\in D$ and $a_1,\dots,a_{2^n}\in D^*$. In other words, $$g_\ell(v_n(x_1,\dots,x_{2^n}), y_1,\dots, y_\ell)=0$$ is a polynomial identity of $D$, so it is also a rational identity of $M_n(F)$ (Lemma~\ref{t2.2}). Note that it is easily seen that $g_\ell(v_n(x_1,\dots,x_{2^n}), y_1,\dots, y_\ell)$ is a non-zero element of $F(x_1,\dots,x_{2^n},y_1,\dots,y_\ell)$~(see~\cite[Theorem 3.4]{hai}). This yields that $$g_\ell(v_n(A_1,\dots,A_{2^n}),B_1,\dots,B_\ell)=0,$$ for all matrices $A_i, B_i\in M_m(F)$. According to Lemma~\ref{2.2}, $v_n(A_1,\dots,A_{2^n})$ is algebraic over $F$ of degree $\le \ell$ for every $A_1,\dots,A_{2^n}\in M_m(F)$.	Now consider the $m\times m$-matrix $T=(t_{ij})_{1\le i,j\le m}$ defined by $t_{i(i+1)}=1$ and $t_{ij}=0$ if $j\ne i+1$. We can show that $tr(T)=0$ and $T$ is algebraic of degree $m$ over $F$. By Lemma~\ref{l4.2}, there exist matrices $A_1,\dots,A_{2^n}\in M_m(F)$ such that $v_n(A_1,\dots,A_{2^n})=T$. Hence, $\ell\ge m$ and this completes the proof.}
\end{proof}

\hspace{-.5cm}\textbf{Acknowledgment:} The research of the first author was supported with ERC [grant number 320974]. The second author was funded by Vietnam National University HoChiMinh City (VNU-HCM) under grant no. C2018-18-03.


\begin{thebibliography}{}

\bibitem{aa} M. Aaghabali, S. Akbari and M.H. Bien, Divison algebras with left algebraic commutators, \textit{Algebr. Represent. Theor.}, \textbf{21} (4) (2018) 807--816.
\vspace{-.2cm}
\bibitem{Pa_AlMu_57} A.A. Albert and B. Muckenhoupt, On mtrices of trace zeros, \textit{Michigan Math. J.}, \textbf{4} (1) (1957) 1--3.
\vspace{-.2cm}
\bibitem{Pa_Am_66} S.A. Amitsur, Rational identities and applications to algebra and geometry, \emph{J. Algebra}, {\bf 3} (1966) 304--359.
\vspace{-.2cm}
\bibitem{amitsur} S.A. Amitsur, L.H. Rowen, Elements of Reduced Trace $0$, \textit{Israel J. Math.}, {\bf 87} (1994) 161--179.
\vspace{-.2cm}	
\bibitem{Bo_BeMaMi_96} K.I. Beidar, W.S. Martindale and A.V. Mikhalev, \textit{Rings with Generalized Identities}, Marcel Dekker, Inc., New York-Basel-Hong Kong, 1996.
\vspace{-.2cm}
\bibitem{chebo} M.A. Chebotar,  Y. Fong and P.H. Lee,  On division rings with algebraic commutators of bounded degree, {\it Manuscripta Math.}, \textbf{113} (2004) 153--164.
\vspace{-.2cm}
\bibitem{chiba} K. Chiba,  Generalized rational identities of subnormal subgroups of skew fields, {\it Proc. Amer. Math. Soc.}, \textbf{124} (6) (1996) 1649--1653.
\vspace{-.2cm}
\bibitem{Bo_Co_71} P.M. Cohn, {\it Free rings and their relations}, Academic Press, New York and London, 1971.
\vspace{-.2cm}
\bibitem{hai} B.X. Hai, T.H. Dung and M.H. Bien, Almost subnormal subgroups in division rings with generalized algebraic rational identities, submitted. (https://arxiv.org/pdf/1709.04774.pdf)
\vspace{-.2cm}
\bibitem{lam}   T.Y. Lam, \textit{A First Course in Noncommutative Rings}, 2nd Ed, GTM,  No. 131, Springer-Verlage, New York, 2001.
\vspace{-.2cm}
\bibitem{m} M. Mahdavi-Hezavehi, Commutators in division rings revisited, \emph{ Bull. Iranian Math. Soc.}, {\bf 26} (2) (2000) 7--88.
\vspace{-.2cm}
\bibitem{mahdavi1} M. Mahdavi-Hezavehi,  Extension of valuations on derived groups of division rings, \emph{Comm. Algebra}, {\bf 23} (3) (1995) 927--940.
\vspace{-.2cm}
\bibitem{mahdavi2} M. Mahdavi-Hezavehi, S. Akbari-Feyzaabaadi, M. Mehraabaadi and H. Hajie-Abolhassan,  Commutators in division rings II, \emph{Comm. Algebra}, {\bf 23} (8) (1995) 2881--2887.
\vspace{-.2cm}
\bibitem{thom} R.C. Thompson, Commutators in the special and general linear groups,  \emph{Trans. Amer. Math. Soc.},
{\bf 101} (1) (1961) 16--33.		
	\end{thebibliography}
\end{document}